\documentclass[11pt,letterpaper]{amsart}
\usepackage{amssymb,amsmath,verbatim,amstext,eucal,amscd}
\usepackage{fullpage}

\usepackage{epsfig}
\RequirePackage{color}

\usepackage{color}

\def\pad{\vskip 4pt}
\long\def\dproof #1\enddproof{\pad\noindent \it Proof. \rm #1 \hfill$\diamondsuit$ \pad}

\newcounter{ccnt}
\long\def\dclaim#1%
\enddclaim{\stepcounter{ccnt}\pad \noindent\textsc{Claim \theccnt:} \it
  #1 \rm}

\long\def\dclaimnn#1%
\enddclaimnn{\pad \noindent\textsc{Claim:} \it
    #1 \rm}

    \numberwithin{equation}{section}

\newtheorem{lemma}[equation]{Lemma}

\newtheorem{corollary}[equation]{Corollary}
\newtheorem{theorem}[equation]{Theorem}
\newtheorem{proposition}[equation]{Proposition}
\newcommand{\FLEX}{\relax}
\newcommand{\flex}[1]{\renewcommand{\FLEX}{#1}}
\newtheorem{flexthm}[equation]{\FLEX}
\newenvironment{flexstate}[2]{\flex{#1}\begin{flexthm}[#2]}{\end{flexthm}}
\newenvironment{flexstate*}[1]{\vskip 5pt \par\noindent{\bf
    #1.}\it}{\rm\vskip 5pt\par}

\newenvironment{dremark*}[1]{\vskip 5pt \par\noindent {\bf #1.}}{\vskip 5pt \par}

\theoremstyle{definition}
\newtheorem{definition}[equation]{Definition}

\newtheorem{example}[equation]{Example}

\newtheorem{remark}[equation]{Remark}

\newcommand{\dstext}[1]{\quad\text{#1}\quad}
\newcommand{\innerprod}[1]{\left\langle #1\right\rangle}

\newcommand{\norm}[1]{\left\|{#1}\right\|}

\DeclareMathOperator{\spn}{span}
\DeclareMathOperator{\ran}{ran}
\DeclareMathOperator{\dom}{dom}

\newcommand{\bh}{\mathcal B(\mathcal H)}

\newcommand{\sH}{\mathcal{H}}

\newcommand{\sK}{\mathcal{K}}
\newcommand{\sL}{\mathcal{L}}

\newcommand{\bbC}{{\mathbb{C}}}

\newcommand{\bbE}{{\mathbb{E}}}

\newcommand{\bbN}{{\mathbb{N}}}

\newcommand{\bbT}{{\mathbb{T}}}

\newcommand{\K}{\mathcal K}

\newcommand{\eps}{\varepsilon}

\newcommand{\cstar}{\hbox{$C^*$}}
\newcommand{\cstaralg}{$C^*$-algebra}

\newcommand{\idealin}{\unlhd}

\title[Approximate Units]{Normalizers and Approximate Units for
  Inclusions of $C^*$-Algebras}

\author[D.R. Pitts]{David R. Pitts}
\address[D.R. Pitts]{
  Department of Mathematics\\
  University of Nebraska-Lincoln\\
  Lincoln\\
  NE 68588-0130 U.S.A.}  \email{dpitts2@unl.edu}


\begin{document}

\begin{abstract}  For an inclusion of \cstaralg s $D\subseteq A$ with
  $D$ abelian, we show that when $n\in A$  normalizes  $D$,
  $n^*n$ and $nn^*$ commute with $D$.   As a corollary, when $D$ is a
  regular MASA in $A$,  every  approximate unit for $D$ is also an
  approximate unit for $A$.
This permits removal of the non-degeneracy hypothesis 
from the definition of a
Cartan MASA in the non-unital case.   

We give examples of singular MASA inclusions: for some, every
approximate unit for $D$ is an approximate unit for $A$, while for
others, no approximate unit for $D$ is an approximate unit for $A$.
Our results imply that if the unitization of an inclusion
$D\subseteq A$ is a \cstar-diagonal, then $D$ is regular in $A$. In
contrast, we give an example of a non-regular inclusion whose
unitization is a Cartan inclusion.

If $D$ is a MASA in $A$, we ask when $A$ is a subalgebra of
$B$ with $D$ a
regular MASA in $B$.  When $D$ is a MASA in $\mathcal B(\ell^2(\bbN))$, no
such $B$ exists.
\end{abstract}
\maketitle

\section{Introduction}
Given an abelian \cstaralg\ $D$, it is often of interest to construct
a larger \cstaralg\ $A$ from data involving $D$.    A very well-known
example occurs when $\Gamma$ is a discrete group of automorphisms  of
$D$ and $A$ is taken to be the completion of the convolution algebra
$C_c(\Gamma, D)$ with respect to a \cstar-norm.   Other examples arise
from studying an inverse semigroup of partial homeomorphisms of $\hat
D$; from this information, one can take $A$ to be a \cstaralg\ arising
from the groupoid of germs of the inverse semigroup.   The resulting
algebra $A$ encodes dynamical properties of the action, which provides
a bridge between topological dynamics and \cstaralg s. 

There are some settings in which the process can be reversed.
Remarkable results of Kumjian~\cite{KumjianOnC*Di}  and
Renault~\cite{RenaultCaSuC*Al} show that given a
\cstaralg\ $A$ and an appropriate maximal abelian $*$-subalgebra
(MASA) $D\subseteq A$, it is possible to use $D$ to introduce
``coordinates'' for $A$, so that $A$ may be described as the
completion of a convolution algebra of continuous functions on a
suitable topological groupoid.  (Many other
authors have produced results along these lines, but Kumjian and
Renault were pioneers.)  An examination of the special case
where $A=D$, shows that the Kumjian-Renault process is a
generalization of the familiar Gelfand representation of $D$ as the
algebra of continuous functions vanishing at infinity on the locally
compact Hausdorff space $\hat D$.

Our interest is with inclusions, which we now define.
\begin{definition}\label{mdef} An \textit{inclusion} is a pair $(A,D)$ where $A$
  and $D$ are \cstaralg s, $D$ is abelian, and $D\subseteq A$.
  The set of \textit{normalizers} for an inclusion $(A,D)$ is
\[N(A,D):=\{n\in A: nDn^*\cup n^*D n\subseteq D\}\] and  $(A,D)$ is
\textit{regular} when the linear span of $N(A,D)$ is norm-dense in
$A$.  
The 
  inclusion $(A,D)$ is \textit{singular} if $N(A,D)=D$.  (When $A$ is
  not abelian, singular MASA 
  inclusions are as far from being regular as possible.)
Finally, when
  $D$ is maximal abelian in $A$, $(A,D)$ is a \textit{MASA inclusion}.

\end{definition}

It is frequently useful to impose a non-degeneracy condition on an
inclusion $(A,D)$.  Often, as in Renault's definition of Cartan
MASA~(\cite[Definition~5.1]{RenaultCaSuC*Al}), this condition is that
the inclusion have the \textit{approximate unit property}, that is,
$D$ contains an approximate unit for $A$. 

When $A$ has a unit and $D$ is a MASA in $A$, the approximate unit
property is automatic.  However, in the non-unital case, it is not
automatic, even when $(A,D)$ is a MASA inclusion,
see~\cite[Section~3.2]{WassermannTePrMaAbSuC*Al} or Example~\ref{big
  C+K} below.  Renault remarks that he included the approximate unit
property property in the definition of Cartan MASA because the
groupoid models he had in mind possess it and, due to
the example in \cite{WassermannTePrMaAbSuC*Al}, it seems to be needed.

The main purpose of this note is to establish Proposition~\ref{posn}
and several of its consequences.  A surprising corollary of
Proposition~\ref{posn} is Theorem~\ref{au}, which shows that when $D$
is a regular MASA in $A$, every approximate unit for $D$ is also an
approximate unit for $A$.  In particular, this shows that the
approximate unit hypothesis may be removed from the definition of a
Cartan MASA, a fact which (understandably) appears to have been
overlooked in the literature.  In~\cite{KumjianOnC*Di}, Kumjian
imposes a different non-degeneracy condition when defining a
\cstar-diagonal in the non-unital setting, but as we note in
Proposition~\ref{NondegK}, it is again automatic.

We give examples showing that in some cases, the regularity
hypothesis in Theorem~\ref{au} may be removed, but as noted above, it
cannot be removed in general.  These examples suggest the
problem of determining which MASA inclusions $(A,D)$ are intermediate to a regular
MASA inclusion $(B,D)$ in the sense that $D\subseteq A\subseteq B$.
Theorem~\ref{nointermed} shows that when $\sH$ is a separable and
infinite dimensional Hilbert space, a MASA inclusion of the form
$(\bh, D)$ is never intermediate to a regular MASA inclusion.

The author is grateful to Anna Duwenig for noticing that the proof
of~\cite[Proposition~3.8]{BrownFullerPittsReznikoffGrC*AlTwGpC*Al}
implicitly used Proposition~\ref{posn}; her observation was the
impetus for the present note.  We thank Jonathan
Brown, Adam Fuller, and Sarah Reznikoff for several helpful conversations and
suggestions.  Finally, we appreciate the referee's useful suggestions
regarding the structure of the paper.

We conclude this section with a pair of preliminary
results, the first of which is folklore.
\begin{flexstate}{Fact}{}\label{thefact}
Let $(A,D)$ be an
  inclusion.  Suppose $\rho$ is a state on $A$ such that $\rho|_D\in
  \hat D$.  Then for any $a\in A$ and $d\in D$,
  \[\rho(ad)=\rho(da)=\rho(a)\rho(d).\]
\end{flexstate}

\begin{proof}
We show that $\rho(ad)=\rho(a)\rho(d)$;  the other equality is
similar.   Let $(u_\lambda)$ be an approximate unit for $A$.
    The Cauchy-Schwarz inequality and the hypothesis
    $\rho|_D\in \hat D$ give
    \begin{align*}
       |\rho(ad-\rho(d)a)|^2&=\lim_\lambda |\rho(ad-\rho(d)au_\lambda)|^2=\lim_\lambda
                              |\rho(a(d-\rho(d)u_\lambda))|^2\\& \leq
    \lim_\lambda \rho(aa^*)\rho((d-\rho(d)u_\lambda)^*(d-\rho(d)u_\lambda))=
    0.
     \end{align*} 
     \end{proof}

For regular inclusions, the following  observation
gives a characterization of the approximate unit property.

\begin{flexstate}{Observation}{}\label{au0} Let $(A,D)$ be an inclusion.
  The following statements hold.
  \begin{enumerate}
  \item If $(A,D)$ has the approximate unit property, then $n^*n\in D$
    for every $n\in N(A,D)$.  
  \item If $(A,D)$ is regular and $n^*n\in D$ for every $n\in N(A,D)$,
    then every approximate unit for $D$ is also an approximate unit
    for $A$. 
  \end{enumerate}
\end{flexstate}

\begin{proof}
  1) Let $(u_\lambda)$ be an approximate unit for $D$ which is also an
  approximate unit for $A$.  For $n\in N(A,D)$,
  $n^*n=\lim_\lambda n^*u_\lambda n\in D$.

  2) 
    This is essentially the argument from the proof
  of~\cite[Lemma~3.10(2)]{BrownFullerPittsReznikoffGrC*AlTwGpC*Al}.
  Fix an approximate unit $(u_\lambda)$ for $D$
and let $n\in N(A,D)$.  As $n^*n\in D$, 
  \begin{equation*}\label{nondeg1}
    (u_\lambda n-n)(u_\lambda n-n)^*=u_\lambda nn^* u_\lambda-nn^*
    u_\lambda -u_\lambda nn^*+nn^*\rightarrow 0,
  \end{equation*}
  whence $u_\lambda n\rightarrow n$.  Replacing $n$ with $n^*$ in this
  argument gives $nu_\lambda \rightarrow n$.  Hence for any
  $a\in \spn N(A,D)$, $u_\lambda a\rightarrow a$ and
  $au_\lambda \rightarrow a$.  Since $\spn N(A,D)$ is dense in $A$,
  $(u_\lambda)$ is an approximate unit for $A$.
\end{proof}

\section{A Commutation Result and Some Consequences}
This section contains our main results: a commutation result,
Proposition~\ref{posn}, and several corollaries.  Among these
consequences are a simpler definition for Cartan inclusions and
clarification of issues surrounding Kumjian's notion of
\cstar-diagonal.

It is not difficult to produce examples of regular inclusions without
the approximate unit property; for instance, take $D$ to be a proper
ideal of the non-unital and abelian \cstaralg\ $A$.  By
Observation~\ref{au0},  there is $n\in N(A,D)$ such
that $n^*n\notin D$.   Our first result shows that for a general inclusion $(A,D)$
and $n\in N(A,D)$, $n^*n$ is intimately related to $D$ despite the
fact that it need not belong to $D$.

\begin{proposition} \label{posn} Let $(A,D)$ be an inclusion.  For $n\in N(A,D)$ and $d\in D$,
  \[n^*nd=d n^*n\in D\dstext{and} nn^*d=dnn^* \in D.\]  Furthermore,
  if $\rho_1$ and $\rho_2$ are states on $A$ such that
  $\rho_1|_D=\rho_2|_D\in \hat D$, then $\rho_1(n^*n)=\rho_2(n^*n)$
  and $\rho_1(nn^*)=\rho_2(nn^*)$.  
\end{proposition}
In the
  terminology of~\cite[Definition~8.2]{ExelPittsChGrC*AlNoHaEtGr}, the
  final statement of Proposition~\ref{posn}  says
  that every element of $\hat D$ is free relative to $n^*n$ and to $nn^*$.

\begin{proof}
Since $N(A,D)$ is closed under adjoints, it suffices to prove the
result for $n^*n$.

  Let $h=n^*n$.  Then $h\in N(A,D)$ because $N(A,D)$ is closed under
  products and the adjoint operation.  For any $d\in D$, $ (d^*hd)^2 =
  d^*(hdd^*h)d\in D$.  Taking the square root
  gives, \begin{equation}\label{mD}
    d^*hd\in D \dstext{for all} d\in D.  \end{equation}

\newcommand{\sot}{\text{\small \sc sot}}

Without loss of generality, we may assume that $A\subseteq \mathcal{B}(\sH)$
for some Hilbert space $\sH$.  Let $Q$ be the orthogonal projection
onto $\overline{D\sH}$ and fix an approximate unit $(u_\lambda)$ for
$D$.  Then 
\[\sot\lim u_\lambda =Q \dstext{and for $d\in D$,} dQ=Qd=d.\]  Using~\eqref{mD} and the fact that on bounded sets, multiplication is jointly continuous in the strong operator topology, we conclude that
$\sot\lim u_\lambda hu_\lambda =QhQ\in D'$ (actually, $QhQ\in D''$).

Furthermore,  
\[Q^\perp hQhQ^\perp=\sot\lim Q^\perp (h u_\lambda h)Q^\perp=0, \]
because $h\in N(A,D)$ and $u_\lambda\in D$.   Therefore, $0=Q^\perp
hQ=QhQ^\perp$, so $Q$ commutes with $h$.  Thus,
\[Qh=hQ=QhQ\in D'.\]

Hence for $d\in D$,
\[
  dh=d(Qh)=(Qh)d=h(Qd)=hd, \] so $d$ commutes with $n^*n$.

Next, for any $0\leq f\in D$,
\[f^2 (n^*n)^2 =(n^*n)f^2(n^*n)\in D\]  because
$n^*n\in N(A,D)$. Therefore, $f^2(n^*n)^2\in D$, so $fn^*n\in D$.
Since $D$ is the span of its positive elements,  $dn^*n\in
D$ for every $d\in D$. 

Finally, suppose for $i=1,2$ that $\rho_i$ are states on $A$ such that
$\rho_1|_D=\rho_2|_D$ is a pure state $\sigma$ on $D$.  Choose $k\in
D$ such that $\sigma(k)\neq 0$.  Applying Fact~\ref{thefact} to $\rho_i(kn^*n)$ gives
$\rho_1(n^*n)=\rho_2(n^*n)$.  

\end{proof}

We now give several corollaries of Proposition~\ref{posn}. Our first 
concerns dynamical objects associated to $n\in N(A,D)$; it 
extends  well-known constructions for  inclusions with the
approximate unit property to arbitrary
inclusions.   With Proposition~\ref{posn} in hand, they are
routine modifications of  those results
(e.g.~\cite[Proposition~6$^\circ$]{KumjianOnC*Di} or~\cite[Propositions~2.1
and~2.2]{PittsStReInI}).
\begin{corollary}\label{dynamics}  Let $(A,D)$ be any inclusion and
 fix  $n\in N(A,D)$. 
  
\begin{description}
    \item[\sc The partial automorphism associated to $n$] Let $B$ be an $AW^*$-algebra with $A\subseteq B$ and let
      $n=u|n|=|n^*|u$ be the polar decomposition of $n$ in $B$.  Then
  $\overline{nn^*D}$ and $\overline{n^*nD}$ are ideals in $D$ and 
      the map $nn^*d\mapsto n^*dn$ uniquely
  extends to a $*$-isomorphism $\theta_n: \overline{nn^*D}\rightarrow
  \overline{n^*nD}$ such that for each $h\in \overline{nn^*D}$,
  \[n\theta_n(h)=hn \dstext{and} u^*hu=\theta_n(h).\] 
\item[\sc The partial homeomorphism associated to  $n$]  Let $\dom n:=\{\sigma\in \hat D: \sigma(n^*n D)\neq 0\}$ and
  $\ran n:= \{\sigma\in \hat D: \sigma(nn^* D)\neq 0\}$.   Then $\dom
  n$ and $\ran n$ are open subsets of $\hat D$ and there is a  homeomorphism
  $\beta_n:\dom n\rightarrow \ran n$ such that for every $h\in
  \overline{nn^*D}$ and $\sigma\in \dom n$,   
\[\beta_n(\sigma)(h)=\sigma(\theta_n(h)).\]   For $\sigma\in\dom n$,
define $\sigma(n^*n):=\rho(n^*n)$, where $\rho$ is any
extension of $\sigma$ to a state on $A$.  Then $\sigma(n^*n)\neq 0$ and 
for $d\in D$,
\[\beta_n(\sigma)(d)=\frac{\sigma(n^*dn)}{\sigma(n^*n)}. \]
\end{description}
\end{corollary}

For a \cstar-algebra $A$, we write $\tilde A$ for its unitization and
$M(A)$ for its multiplier algebra.  When $(A_1, D_1)$ and $(A_2, D_2)$
are inclusions, a $*$-homomorphism $\alpha: A_1\rightarrow A_2$ is
\textit{regular} if $\alpha(N(A_1,D_1))\subseteq N(A_2,D_2)$.  We use
the notation $\alpha:(A_1,D_1)\rightarrow (A_2,D_2)$ when $\alpha$ is
a regular $*$-homomorphism.  Our next corollary concerns the
regularity of the inclusion map of $A$ into $\tilde A$ or another
unital \cstar-subalgebra of $M(A)$.

\begin{corollary}\label{regmap}  Suppose $(A,D)$ is a MASA 
  inclusion and $A$ is not unital.
  \begin{enumerate}
    \item 
  The usual embedding $\iota: A\rightarrow \tilde A$ is a
  regular $*$-monomorphism of $(A,D)$ into $(\tilde A, \tilde D)$.
\item Suppose in addition that $(A,D)$ has the approximate unit  property and that $B$ is a unital \cstar-algebra containing $A$ as an essential ideal.   Let $D_B:=\{b\in B: bD\cup Db\subseteq D\}$.  Then $(B,D_B)$ is a MASA inclusion and $N(A,D)\subseteq N(B,D_B)$.  In fact, for every $n\in N(A,D)$ and $b\in D_B$,
  \[nbn^*\in D\dstext{and} n^*bn\in D.\]
\end{enumerate}
\end{corollary}
\begin{proof}  1) For $(d,\lambda)\in \tilde D$ and $n\in N(A,D)$,
  \[\iota(n)(d,\lambda)\iota(n^*)= (n,0)(d,\lambda)(n^*,0)=(ndn^*
    +\lambda nn^*,0). \]   As $D$ is a MASA, Proposition~\ref{posn}
  gives $nn^*\in D$, so $\iota(n)\tilde D\iota(n)^*\in \tilde D$.
  Similarly $\iota(n)^*\tilde D\iota(n)\in \tilde D$, so $\iota(n)\in
  N(\tilde A,\tilde D)$. 

  2) We may assume $A\subseteq \bh$ is a non-degenerate \cstar-algebra
  and $M(A)=\{x\in \bh: xA\cup Ax\subseteq A\}$.  Since $A$ is an
  essential ideal of $B$, 
  \[I_B=I_\sH\in B\subseteq M(A).\] Routine arguments and the
  approximate unit property yield: $D\subseteq \bh$ is non-degenerate;
  $M(D)\subseteq M(A)$; and $D_B=B\cap M(D)$ is a MASA in $B$.

Let $n\in N(A,D)$ and $b\in D_B$.  Note that $nb$
belongs to $N(A,D)$.    Since $D$ is a MASA in $A$,
Proposition~\ref{posn} gives $nbb^*n^*\in D$.   Similarly,
$n^*b^*bn\in D$.   Since $D_B$ is spanned by its positive elements, 
$nD_B n^*\cup n^*D_B \, n\subseteq D$.

\end{proof}

  Recall that $n\in N(A,D)$ is \textit{free} if $n^2=0$;
  we write $N_f(A,D)$ for the collection of free normalizers.

\begin{corollary}\label{Nfree}  Suppose $(A,D)$ is an inclusion and
    $n\in N_f(A,D)$.  If $\rho$ is a state on $A$ such
    that $\rho|_D$ is a pure state on $D$, then $\rho(n)=0$.
  \end{corollary}
  \begin{proof}
Proposition~\ref{posn} gives $n^*nd=dn^*n\in D$ and
$nn^*d=dnn^*\in D$, so if $d\in D$ satisfies $\rho(d)=1$,  Fact~\ref{thefact} gives
\[\rho(n^*n)\rho(nn^*)=\rho(dn^*n)\rho(nn^*d)=\rho(dn^*n^2n^*d)=0,\]
whence $0\in \{\rho(n^*n), \rho(nn^*)\}$.    

With $d\in D$ again satisfying $\rho(d)=1$,
$|\rho(n)|^2=|\rho(dn)|^2=|\rho(nd)|^2$.  By the Cauchy-Schwartz
inequality,  \[|\rho(n)|^2\leq \min\{\rho(n^*n), \rho(nn^*)\}=0.\]

\end{proof}

\begin{theorem} \label{au} 
If $(A,D)$ is a regular  MASA inclusion, then every approximate unit
  for $D$ is an approximate unit for $A$.
  \end{theorem}
\begin{proof}
  Proposition~\ref{posn} shows that if $n\in N(A,D)$, then $n^*n$
  commutes with $D$.  As $D$ is a MASA in $A$,
  $n^*n\in D$.   Now apply Observation~\ref{au0}. 
\end{proof}

An immediate consequence of Theorem~\ref{au} is that the approximate
unit property behaves well for intermediate subalgebras of regular MASA inclusions.  
  \begin{corollary}\label{intermediate}  Suppose $(A,D)$ is a regular MASA inclusion and
    $B$ is a norm-closed, but not necessarily selfadjoint, subalgebra
    satisfying $D\subseteq B\subseteq A$.  Then every approximate unit
    for $D$ is an approximate unit for $B$.
  \end{corollary}

We shall use the following definition of Cartan inclusion in
the sequel.    By Theorem~\ref{au}, it is equivalent to
Renault's original definition of Cartan inclusion. 
\begin{definition}[cf.\ {\cite[Definition~5.1]{RenaultCaSuC*Al}}] \label{Cdef} A \textit{Cartan inclusion} is a regular MASA
inclusion $(A,D)$ such that there exists a faithful conditional expectation $P:
A\rightarrow D$.   
\end{definition}

\begin{remark} \label{rMASAr} The class of regular MASA inclusions is
    considerably broader than the class of Cartan inclusions.  Indeed,
    given a  regular MASA
inclusion $(A,D)$ with $A$ unital, there is always a unique pseudo-expectation $\bbE$ 
of $A$ into the injective envelope $I(D)$ of $D$ (see
\cite[Definition~1.3]{PittsStReInI} and \cite[Theorem~3.5]{PittsStReInI}), and $(A,D)$ falls
into exactly one of the following four  classes of unital
regular MASA inclusions: 
\begin{itemize}
  \item[i)] $\bbE$ is a faithful conditional expectation of $A$ onto $D$
    (in this case $(A,D)$ is a Cartan inclusion);
    \item[ii)] $\bbE$ is a non-faithful conditional expectation of $A$ onto $D$;
      \item[iii)] $\bbE$ is faithful, and there is no conditional expectation of $A$ onto
        $D$;
      \item[iv)] $\bbE$ is not faithful, and
        there is no conditional expectation of $A$ onto
      $D$.
    \end{itemize}
    These classes are non-void, and we now describe or
    reference constructions for each.  The
    simplest non-trivial example of a
    Cartan inclusion is
    obtained by taking $D$ to be a MASA in $A=\mathcal B(\bbC^n)$, but
    far more interesting examples are abundant in the literature.
    See~\cite[Theorem~2.2 and
    Example~2.3]{ExelPittsZarikianExIdReFrTrGr} for a construction of
    a family of examples in class
    ii).  A crossed product construction
    found in~\cite[Section~6]{PittsStReInI} produces examples in class
    iii).  Finally, the direct sum of two inclusions, with one belonging to 
    class ii) and the other in class iii), will yield an example in class iv).
  \end{remark}

We wish to discuss Kumjian's notion of \cstar-diagonal in
light of our results so far.  After Renault's
papers~\cite{RenaultTwApDuGrC*Al,RenaultCaSuC*Al} appeared, it became commonplace to define a
\cstar-diagonal as an inclusion $(A,D)$ such that $D$ is a Cartan MASA
in $A$ having the \textit{pure state extension property},  in the sense
of~\cite[Definition~2.5]{ArchboldBunceGregsonExStC*AlII}, that is,
every pure state of $D$ extends uniquely to a pure state of $A$ and no
pure state of $A$ annihilates $D$.  (The pure state extension property
is also called 
the \textit{extension property}.)

This definition differs from Kumjian's original definition of
  \cstar-diagonal.  Indeed, in his original paper on the topic, Kumjian
(\cite[Definition~3$^\circ$]{KumjianOnC*Di}) says an inclusion $(A,D)$
with $A$ unital is a \cstar-diagonal if
\begin{enumerate} \item[(I)] there is a faithful
  conditional expectation $P:A\rightarrow D$; and
  \item[(II)] $\spn
N_f(A,D)$ is dense in $\ker P$.   
\end{enumerate}
In the non-unital setting,  Kumjian defines an
inclusion $(A,D)$  to be a \cstar-diagonal if its
unitization $(\tilde A, \tilde D)$ satisfies Conditions (I) and (II).

Let us say that an inclusion $(A,D)$ (where $A$ is not assumed unital)
\textit{satisfies Kumjian's
conditions} if both (I) and (II) hold.

In the unital case, the two notions for \cstar-diagonal just described
are known to be equivalent, but we do not know of a reference where
this fact is established.  For convenience, we therefore outline a
proof in the first part of Proposition~\ref{NondegK} below.

When $A$ is not unital, it is unclear whether an inclusion $(A,D)$
satisfying Kumjian's conditions must be a \cstar-diagonal in the sense
that $(\tilde A, \tilde D)$ satisfies Kumjian's conditions or a
\cstar-diagonal in the
sense that $(A,D)$ is a Cartan inclusion with
the extension property.
The next result shows there is no ambiguity:  if $(A,D)$ satisfies
Kumjian's conditions, then $(\tilde A,
\tilde D)$ satisfies Kumjian's conditions and $(A,D)$ is a Cartan inclusion with
the extension property.

\begin{proposition}\label{NondegK}  Let $(A,D)$ be an inclusion.
  \begin{description}
    \item[\sc The Unital Case]  If $A$ is unital, then $(A,D)$ satisfies
      Kumjian's conditions if and only if $(A,D)$ is a Cartan
      inclusion and every pure state on $D$ uniquely extends to a
      state on $A$.
  
\item[\sc The Non-Unital Case]
      Suppose  $A$ is not unital.  The following are equivalent.
  \begin{enumerate}
  \item $(A,D)$ satisfies Kumjian's conditions.
    \item $(A,D)$ is a Cartan inclusion such that every pure state of
      $D$ has a unique extension to a  state on $A$.
      \item $(A,D)$ is a Cartan inclusion such that every pure state of
      $D$ has a unique extension to a  state on $A$ and no pure
      state of $A$ annihilates $D$.
\item $(\tilde A, \tilde D)$ is a Cartan inclusion such that every
  pure state of $\tilde D$ extends uniquely to a  state on $\tilde A$.
          \item $(\tilde A,\tilde D)$ satisfies Kumjian's conditions.
          \end{enumerate}
        \end{description}
      \end{proposition}

\begin{proof}
  \sc The Unital Case. \rm Suppose $A$ is unital.

  It follows from~\cite[Proposition~4$^\circ$]{KumjianOnC*Di} that if
  $(A,D)$ satisfies Kumjian's conditions, then it is a Cartan
  inclusion and every pure state on $D$ uniquely extends to a state on
  $A$.  (Throughout~\cite{KumjianOnC*Di}, Kumjian makes the blanket
  assumption that all \cstaralg s are separable, but his proof
  of~\cite[Proposition~4$^\circ$]{KumjianOnC*Di} is also valid when
  $A$ is not separable.)

  Now suppose $(A,D)$ is a Cartan inclusion such that every pure state
  on $D$ extends uniquely to a state on $A$.  Let $ P$ be the
  conditional expectation of $ A$ onto $ D$ and let $n\in N( A, D)$.
  An application of~\cite[Proposition~3.10]{DonsigPittsCoSyBoIs} (with
  $x=I$) shows
  that $ \{n P(n^*), n^*P(n)\} \subseteq D$.   A computation now
  yields $n-
  P(n)\in N(A,D)$.
Since $(A,D)$ is a regular inclusion and $ P$ is contractive, to show that
  $\spn N_f( A,  D)$ is dense in $\ker  P$, it is
   enough to show that  $N( A, D)\cap \ker
  P\subseteq  \overline{\spn} N_f( A, D)$.

  Fix $n\in N( A, D)\cap \ker P$.  For
  $\sigma\in \dom n$,~\cite[Proposition~3.12]{DonsigPittsCoSyBoIs}
  or~\cite[Lemma~9$^\circ$]{KumjianOnC*Di} gives 
\[\beta_n(\sigma)\neq \sigma.\]  To show
$n\in \overline{\spn} N_f(A,D)$ use a partition
of unity argument.  Here are some of the details.

Let $X$ be the (compact) set of all pure states on $ D$ and for
$\eps >0$, let $X_\eps:=\{\sigma\in X: \sigma(n^*n)\geq \eps^2\}$.
For  $\sigma\in X_\eps$, the fact that
$\beta_n(\sigma)\neq \sigma$ implies that we may find
$a_\sigma\in \overline{nn^*D}$ such that:  $\sigma(a_\sigma)\neq 0$,
$0\leq a_\sigma\leq 1$, and $a_\sigma\theta_n(a_\sigma)=0$.  Using 
compactness of $X_\eps$,  select a finite subset
$\{b_j\}_{j=1}^k\subseteq \{a_\sigma: \sigma\in X_\eps\}$ such that
the Gelfand transform of $b:=\sum_{j=1}^k b_j$ does not vanish on $X_\eps$.  If
$m=\min\{\sigma(b): \sigma\in X_\eps\}$ and
$Z:=\{\sigma\in X: \sigma(b)\leq m/2\}$, then
$X_\eps \cap Z=\emptyset$.  If $Z\neq \emptyset$, use
Urysohn's lemma to find $h\in D$ with $0\leq h\leq 1$ such that
$\hat h|_Z=0$ and $\hat h|_{X_\eps}=1$;  if $Z=\emptyset$, let $h=1$.  For $1\leq j\leq k$ and
$\sigma\in X$, let
  \[f_j(\sigma)=\begin{cases} \sigma(h) \displaystyle
      \frac{\sigma(b_j)}{\sigma(b)} & \sigma\notin Z \\ 0& \sigma\in
      Z.\end{cases}\] Then $f_j\in C(X)$, so there exists $d_j\in D$
  with $\hat d_j=f_j$.  We have produced a collection
  $\{d_j\}_{j=1}^k\subseteq \overline{nn^*D}$ such that for each $j$,
  $0\leq d_j \leq 1$ and $d_j\theta_n(d_j)=0$.
Hence $d_jn\in N_f(A,D)$ for   $1\leq j\leq
k$. 
  Furthermore, 
  $d:=\sum_{j=1}^k d_j$ satisfies
  $0\leq d\leq 1$ and $\sigma(d)=1$ for each $\sigma\in X_\eps$.
 Then
\[\norm{dn-n}^2=\sup_{\sigma\in X} \sigma((d-I)^2nn^*)
  =\sup_{\sigma\in (X\setminus X_\eps)} \sigma((d-I)^2nn^*)
  <\eps^2.\]  As $dn=\sum_{j=1}^k d_j n\in \spn N_f(A,D)$, we conclude
$n\in \overline{\spn}N_f(A,D)$.
Thus  the
unital case holds.

\vskip 6pt

\noindent \sc The Non-Unital Case. \rm
For the remainder of the proof,  assume $A$ is not unital. 

(1) $\Rightarrow$ (2).
Since $A=D+\ker P$, $(A,D)$ is regular.  Corollary~\ref{Nfree} and
condition (II) imply that every pure state on $D$ extends uniquely to
a (necessarily pure) state on $A$.  We cannot immediately conclude
that $D$ is a MASA in $A$ because of the possibility that there is a
pure state on $A$ which annihilates $D$.  Condition (I) is key to
showing $D$ is a MASA.  \dclaimnn $D$ is a MASA in $A$.  \enddclaimnn
\dproof Let $D_1\subseteq A$ be a MASA containing $D$ and let
$S:=\{\rho\in\hat D_1: \rho|_D\neq 0\}$.  Suppose $a\in D_1$ and
$\rho(a)=0$ for every $\rho\in\hat D_1\setminus S$.  The fact that
every pure state of $D$ extends uniquely to a pure state on $D_1$
implies that for every $\rho\in S$, $\rho(a)=\rho(P(a))$, and clearly
if $\rho\in \hat D_1\setminus S$, $\rho(a)=0=\rho(P(a))$. 
Therefore, $a=P(a)$, that is, $a\in D$.  It follows that
\[D=\{a\in D_1: \rho(a)=0 \text{ for all } \rho\in \hat D_1\setminus
  S\},\] so $D$ is an ideal in $D_1$.  If $0\leq h\in D_1$ and
$hd=0$ for all $d\in D$, then $0\leq P(h)^2 =P(hP(h))=0$.  Thus
$P(h)=0$, so $h=0$ by faithfulness of $P$.  It follows that $D$ is an
essential ideal in $D_1$.

Fix $a\in D_1$.   Then for any $d\in D$,
$(a-P(a))d=ad-P(ad)=0$.   As $D$ is an essential ideal,  $a=P(a)$.
Therefore $D=D_1$,  so $D$ is a MASA in $A$.
\enddproof

Thus $(A,D)$ is a Cartan inclusion such that each pure state of $D$ extends
uniquely to a pure state of $A$.

(2) $\Rightarrow$ (3).  Let $(u_\lambda)$ be an approximate unit for
$D$.  Theorem~\ref{au} ensures that $(u_\lambda)$ is an approximate
unit for $A$.  Thus, for any pure state $\rho$ on $A$, 
$\lim_\lambda \rho(u_\lambda)=1$, so (3) holds.

(3) $\Rightarrow$ (4).   
It is readily seen that the map
  $\tilde P: \tilde A\rightarrow \tilde D$ given by
  $\tilde P(a,\lambda)=(P(a),\lambda)$ is  a faithful conditional
  expectation.   That $(\tilde A, \tilde D)$ is regular follows from 
  Corollary~\ref{regmap},
  and~\cite[Remarks~2.6(iii)]{ArchboldBunceGregsonExStC*AlII}  gives
  the extension property for $(\tilde A, \tilde D)$, so  (4) holds.

(4) $\Rightarrow$ (5).  This follows from the unital case.

(5) $\Rightarrow$ (1).  Let $\tilde P: \tilde A\rightarrow \tilde D$
be the faithful conditional expectation.  As in the proof of (1)
implies (2) above,
Corollary~\ref{Nfree} and Condition (II) show that
$(\tilde A, \tilde D)$ has the extension property
(\cite[Proposition~4$^\circ$]{KumjianOnC*Di} also gives this).
Therefore $(\tilde A, \tilde D)$ is a MASA inclusion.  Hence $(A, D)$
is also a MASA inclusion.

If $v:=(a,\lambda)\in N_f(\tilde A,\tilde D)$,
 then $\lambda=0$ and $a\in N_f(A,D)$.  This shows
 $N_f(\tilde A,\tilde D)\subseteq \iota(N_f(A,D))$.  
 Part (1) of Corollary~\ref{regmap} gives $\iota(N_f(A,D))\subseteq N_f(\tilde A,
 \tilde D)$.  Thus,  
 $N_f(\tilde A, \tilde D)=\iota(N_f(A,D))$, whence 
  \[\iota(\overline{\spn} N_f(A,D))=\overline{\spn} N_f(\tilde
A,\tilde D)=\ker \tilde P.\]
Since 
\[\tilde A= \ker\tilde P + \iota(D) + \bbC\, (0,1)\] and   $\ker \tilde
P\subseteq \iota(A)$,  $\tilde P$ leaves $\iota(A)$ invariant. 
Therefore, $P:=\iota^{-1}\circ \tilde
P\circ \iota$ is a
faithful conditional expectation of $A$ onto $D$.  Since $\iota(\ker
P)=\ker \tilde P=\iota(\overline{\spn} N_f(A,D))$, we obtain $\ker
P=\overline\spn N_f(A,D)$.   Thus $(A,D)$ satisfies Kumjian's
conditions and the proof is complete.
  
\end{proof}

Proposition~\ref{NondegK} shows the following definition of
\cstar-diagonal is equivalent to Kumjian's original definition,
regardless of whether $A$ is unital.
\begin{definition} An inclusion $(A,D)$ is a \textit{\cstar-diagonal}
  if it satisfies Kumjian's conditions, or equivalently, if $(A,D)$  is a
  Cartan inclusion for which every pure state of $D$ uniquely extends
  to a state on $A$.
\end{definition}
\section{Examples}

  The purpose of this section is to give a variety of examples.  
We begin with an example concerning the unitization of an inclusion.

Let $(A,D)$ be an inclusion with $A$
non-unital. Proposition~\ref{NondegK} shows $(A,D)$ is a
\cstar-diagonal if and only if its unitization $(\tilde A, \tilde D)$
is a \cstar-diagonal.  This is easily seen to be true for the class of
MASA inclusions, but as we shall see momentarily, it is not true for the class of Cartan
inclusions.  While a routine argument (sketched in the proof of
Proposition~\ref{CiffUC}) shows that
$(\tilde A,\tilde D)$ is a Cartan inclusion whenever $(A,D)$ is
Cartan, our first example shows the converse fails: it is possible for the
unitization of a non-regular inclusion $(A,D)$ to be a Cartan
inclusion.

\begin{example}\label{UnitalCartan}   Let $S$ be the unilateral shift,
  let $B=C^*(S)$ be the Toeplitz algebra, let
  $D_1=C^*(\{I\}\cup \{S^n S^{*n}: n\in \bbN\})$, and let $\K$ be the
  compact operators.  Then $(B,D_1)$ is a Cartan inclusion.

  Identify $B/\K$ with $C(\bbT)$, and let $q:B\rightarrow
  C(\bbT)$ be the quotient map.  Fix $\omega\in \bbT$ and for $b\in
  B$, let $\tau_\omega(b)=q(b)(\omega)$.  Then $\tau_\omega$ is a
  multiplicative linear functional on $B$.  Let
  \[A=\ker\tau_\omega\dstext{and} D= D_1\cap A.\] The map
  $\tilde A\ni (x,\lambda)\mapsto x+\lambda I\in B$ is an isomorphism
  which carries $\tilde D$ onto $D_1$.   Thus $(\tilde A,\tilde D)$ is a
  Cartan inclusion (and in particular a MASA inclusion).  Since $D\subseteq \K$ and $A$
  contains the non-compact operator $S-\omega I$, $(A,D)$ does not
  have the approximate unit property.   Theorem~\ref{au}
  implies $(A,D)$ cannot be a regular inclusion, so $(A,D)$ is not a
  Cartan inclusion.

\end{example}
We now note that the behavior displayed in
Example~\ref{UnitalCartan} cannot occur when $(A,D)$ is regular.

\begin{proposition}\label{CiffUC}  Suppose $(A,D)$ is a regular
  inclusion, with $A$ not unital.  Then $(A,D)$ is a Cartan inclusion if and only
  if $(\tilde A,\tilde D)$ is a Cartan inclusion.
\end{proposition}
\begin{proof} We again use $\iota: A\rightarrow \tilde A$ for
  the map $x\mapsto (x,0)$.  Suppose $(\tilde A,\tilde D)$ is a Cartan
  inclusion.  Then $(A,D)$ is a regular MASA inclusion, so  
  Theorem~\ref{au} ensures $D$ contains an approximate unit
  $(u_\lambda)$ for $A$.  Let $\tilde P:\tilde A\rightarrow \tilde D$
  be the conditional expectation.  For $x\in A$,
  \[\tilde P(\iota(x))=\lim \tilde P(\iota(u_\lambda x))=\lim \iota(u_\lambda)\tilde
    P(\iota(x)).\]  Thus $\iota(A)$ is invariant
  under $\tilde P$, and, as $\iota$ is one-to-one, 
  $P:=\iota^{-1} \circ \tilde P \circ \iota$ is a conditional
  expectation of $A$ onto $D$.  As $\tilde P$ is faithful, so is $P$, whence
  $(A,D)$ is a Cartan inclusion.

  We sketch the converse.  Suppose $(A,D)$ is a Cartan inclusion with
  conditional expectation $P: A\rightarrow D$. Then
  $(\tilde A,\tilde D)$ is a regular MASA inclusion.  Define
  $\tilde P: \tilde A\rightarrow \tilde D$ by
  $\tilde P(x,\lambda)=(P(x),\lambda)$.  For
  $(x,\lambda)\in \tilde A$, the fact that $P(x^*x)\geq P(x)^*P(x)$
  gives
  \[P((x,\lambda)^*(x,\lambda))\geq (P(x)^*P(x)+ \overline\lambda x+\lambda
    x^*, |\lambda|^2) =\tilde P(x,\lambda)^*\tilde P(x,\lambda)\geq
    0.\] Then $\tilde P$ is a faithful conditional
  expectation, so 
  $(\tilde A,\tilde D)$ is a Cartan inclusion. 
 
\end{proof}

Next we
  give examples of singular MASA
  inclusions, some of which have the approximate unit property, while
  others do not.  Corollary~\ref{intermediate} is our key tool for
  constructing non-regular inclusions with the approximate unit
  property: we make appropriate choices of subalgebras intermediate to
  a regular MASA inclusion.  Theorem~\ref{nointermed} gives an example
  of a MASA inclusion which is not intermediate to a regular MASA
  inclusion.    These results lead us to pose a number of questions.

  If $(A,D)$ is a MASA inclusion and $B$ is an intermediate
  \cstar-subalgebra,  $D\subseteq B\subseteq A$, regularity
  properties of $(B,D)$ cannot be deduced from regularity of $(A,D)$.
  Indeed, $(B,D)$ may be
  \begin{enumerate}
    \item[(a)] regular: take $D$ to be the $n\times n$ diagonal
      matrices and $D\subseteq B\subseteq M_n(\bbC)$;
      \item[(b)] singular: take $D$
  to be a non-atomic MASA in $\bh$, let $A=\overline\spn N(\bh, D)$ and
  $B=D+\sK$,
  see~\cite[Corollary~3.9 and Proposition~2.9]{ExelPittsZarikianExIdReFrTrGr}; or
  \item[(c)] 
  something peculiar: 
  \cite[Example~5.1]{BrownExelFullerPittsReznikoffInC*AlCaEm} gives
  an example of a Cartan pair $(A,D)$ and an intermediate non-regular
  \cstar-subalgebra $B$ having full support in the Renault twist
  $\Sigma\rightarrow G$ associated with $(A,D)$.
\end{enumerate}
In each of the latter two 
  cases, $(B,D)$ is a non-regular MASA inclusion with the
  approximate unit property.

  Before continuing, we give a method for constructing
  singular MASA inclusions.  Interestingly, the proof in the
  non-unital case uses Corollary~\ref{regmap}.  We use the notation
  $J\idealin B$ to indicate $J$ is a norm-closed, two-sided ideal in
  $B$.
\begin{lemma}\label{singular}
Let $(A,D)$ be an inclusion with $D$ a MASA in $A$.   If $J\idealin A$
satisfies $J\cap D=\{0\}$, then $(D+J, D)$ is a singular inclusion.
\end{lemma}
\begin{proof}
For notational purposes, let $B=D+J$. 
The case when $B$ is unital is
\cite[Proposition~2.9]{ExelPittsZarikianExIdReFrTrGr}.

Now suppose $B$ is not
unital.   Then
$\tilde D$ is a MASA in $\tilde B$, $J\idealin \tilde B$, and $J\cap
\tilde D=J\cap D=\{0\}$.      
As $D$ is a MASA in $B$, Corollary~\ref{regmap} and the unital case give
\[N(B,D)\subseteq N(\tilde B,
\tilde D)=\tilde D.\] 
Thus, $N(B,D)\subseteq B\cap \tilde D=D$, so $(B,D)$ is a singular inclusion.

\end{proof}

\begin{example}\label{C+K}
  Here is a class of  examples of  singular MASA inclusions having the
  approximate unit property.  Let $\sH$ be a separable infinite
  dimensional Hilbert space, and suppose $D\subseteq \bh$ is a
  \cstaralg\ whose double commutant is a non-atomic MASA in $\bh$.
  Letting $\sK\subseteq \bh$ be the compact operators, note that
  $D\cap \sK\subseteq D''\cap \sK=\{0\}$.  Set $A=D+\sK$.  Then $D$ is
  a MASA in $A$ and Lemma~\ref{singular} shows it is singular in $A$.
  If $(u_\lambda)$ is an approximate unit for $D$, then $\text{\sc
    sot}\lim u_\lambda=I$, so for every $K\in \sK$, the
  nets $(u_\lambda K)$ and $(Ku_\lambda)$ norm-converge to $K$.  It
  follows that $(u_\lambda)$ is an approximate unit for $A$.
\end{example}

We next present our example of a singular MASA inclusion $(A,D)$, with
$A$ separable, such that no approximate unit for $D$ is an approximate
unit for $A$.  While the approximate unit property also fails for the MASA
inclusion $(A_0,C_0)$ found in
\cite[Section~3.2]{WassermannTePrMaAbSuC*Al}, that example differs
significantly from ours: $C_0$ is generated by minimal
projections, the expectation of $A_0$ onto $C_0$ is faithful,  and
$C_0$ is not singular in $A_0$.

\begin{example}  \label{big C+K}
  Given any \cstaralg\ $\mathfrak A$, $\ell^\infty(\mathfrak A)$ will
  denote the \cstaralg\ of all bounded sequences in $\mathfrak A$
  (with the usual pointwise operations and supremum norm).
  \providecommand{\V}{\mathcal V}
  
  Let $\sH=L^2[0,1]$ (Lebesgue measure) and fix a  set of vectors
  \[\{\xi_j: j\in\bbN\}\subseteq \sH\setminus\{0\} \dstext{such that}
  \overline{\{\xi_j\}}_{j\in\bbN}=\sH.\]
  For each $j\in \bbN$, let $p_j\in\bh$ be the rank-one projection onto $\bbC
  \xi_j$.  Notice that for any $x\in\bh$,
  \[\norm{x}=\sup_j \norm{xp_j}=\sup_j\norm{p_jx}.\]
  Let $M$ be the collection of all multiplication operators,
  \[\sH\ni \xi\mapsto f\xi, \dstext{where} f\in C_0(0,1),\] and let $p$ be the
  projection in $\ell^\infty(\sK)$ whose $j$th term is $p_j$,
  that is,
  \[p: j\mapsto p_j.\]

  We now describe the inclusion.  Take $D$  to be the set of all constant
sequences in $\ell^\infty(M)$, and define
\[A=C^*(\{p\}\cup D)\subseteq \ell^\infty(M+\sK).\]
$A$ is separable because $D\simeq M$.
\dclaimnn
$D$ is a singular MASA in $A$.
\enddclaimnn
\dproof
 Since $M\cap \sK=(0)$, the map $\Phi:
M+\sK\rightarrow M$ given by $\Phi(m+k)=m$ (where $m\in M$ and $k\in \sK$),
is a well-defined  $*$-homomorphism.   Next, let $\Delta:
\ell^\infty(M+\sK)\rightarrow \ell^\infty(M)$ be the $*$-epimorphism given by
\[(\Delta(x))(j)=\Phi(x(j))\quad j\in \bbN.\]

Let us show that $\Delta(A)=D$.  Since $\Delta|_D=\text{id}|_D$, it
suffices to show $\Delta(A)\subseteq D$.  To do this, let $X$ be the
collection of all finite products where each factor is taken from
$\{p\}\cup D$ and at least one of the factors is $p$.  Let
\[Y=\spn X\subseteq A,\] and note that for $y\in Y$ and $j\in \bbN$,
$y(j)\in \sK$.  The definitions of $\Delta$  and $A$ show that
$Y\subseteq \ker \Delta$ and $D+Y$ is dense in $A$.  Therefore, given
$a\in A$ and $\eps>0$, we may find $d\in D$ and $y\in Y$ so that
$\norm{a-(d+y)}<\eps$.  As $\Delta(d+y)=d$, we obtain $\norm{\Delta(a)
  -d}<\eps$, showing that $\Delta(a)$ can be approximated as closely
as desired by an element of $D$.  Therefore, $\Delta(a)\in D$.  

Now suppose $a\in A$ commutes with $D$.  Then for each $j\in\bbN$, $a(j)$
commutes with $M$.  As $M$ is a
MASA in $M+\sK$, $a(j)\in M$.  This gives $a=\Delta(a)\in D$.   Therefore,
$D$ is a MASA in $A$. 

Since $\Delta|_A$ is a homomorphism, $J:=\ker(\Delta|_A)$ is an ideal
of $A$ satisfying $D\cap J=(0)$.  Also $D+J=A$ (because for $a\in A$,
$a=\Delta(a)+(a-\Delta(a)$).  An application of Lemma~\ref{singular}
shows $(A,D)$ is a singular MASA inclusion.  
\enddproof

Now suppose $(u_\lambda)$ is an approximate unit for $D$.  Let
$v_\lambda =u_\lambda(1)\in M$ (recall elements of $D$ are constant
sequences in $M$).  Note that $\norm{v_\lambda -I_\sH}=1$ because
$v_\lambda$ is the multiplication operator determined by an element
$f_\lambda\in C_0(0,1)$ with $0\leq f_\lambda$ and
$\norm{f_\lambda}\leq 1$.  Fixing $\lambda$, we find
\[
  \norm{pu_\lambda-p}=\sup_j\norm{p_jv_\lambda-p_j}=\sup_j\norm{p_j(v_\lambda
    -I_\sH)}=\norm{v_\lambda-I_\sH}=1.\]  Therefore,
$(u_\lambda)$ is not an approximate unit for $A$.
\end{example}

Examples~\ref{C+K} and~\ref{big C+K} and the 
unpredictable behaviour of regularity for intermediate inclusions (see
the discussion 
between~Example~\ref{UnitalCartan} and Lemma~\ref{singular}) motivate
the following.
\begin{flexstate}{Question}{} \label{aup}
  Suppose $D\subseteq A$ is a MASA inclusion.  Under
  what circumstances is there a \cstaralg\ $B$ with $D\subseteq
  A\subseteq B$ such that $D$ is a regular MASA in $B$?  
\end{flexstate}

Corollary~\ref{intermediate} shows the approximate unit property is a
necessary condition for $(A,D)$ to be an intermediate inclusion
arising from a regular MASA inclusion $(B,D)$.  In particular, the
inclusion of Example~\ref{big C+K} fails to be such an intermediate
inclusion.  The inclusion described in item (b) at the start of the
present section is intermediate to a regular inclusion, so some
choices for $D$ in Example~\ref{C+K} yield singular inclusions
intermediate to a regular MASA inclusion.  We do not know whether
every inclusion described in Example~\ref{C+K} is intermediate to a
regular MASA inclusion, however for such an inclusion $(D+\sK,D)$,
Proposition~\ref{posn} suggests that
$\overline{\spn} \{nd: n\in N(\bh, D), \, d\in D\}$ might be a
plausible candidate for the regular algebra $B$.

There are MASA inclusions $(A,D)$ with $A$ unital for which $A$ cannot
be an intermediate subalgebra of a regular MASA inclusion $(B,D)$.  In
fact, it can happen that $A$ contains no MASA $D$ with this property.
The following result gives an example.

\begin{flexstate}{Theorem}{}\label{nointermed}
Let $\sH$ be a separable, infinite dimensional Hilbert space and
 suppose $D$ is a MASA in $\bh$.  There is
 no regular MASA inclusion $(B,D)$ with $D\subseteq \bh \subseteq B$.
\end{flexstate}
\begin{proof}
  Let $P\in D$ be  the strong operator topology sum of the minimal projections of
 $D$.

 Suppose first that $P<I$. Since $PD$ is an atomic MASA in
 $\mathcal B(P\sH)$, there is a unique (and faithful) conditional
 expectation $E_P: \mathcal B(P\sH)\rightarrow PD$.  Also, $P^\perp D$
 is a non-atomic MASA acting on $\mathcal B(P^\perp \sH)$.  By
 \cite[Theorem~2]{KadisonSingerExPuSt} there are multiple conditional
 expectations of $\mathcal B(P^\perp\sH)$ onto $P^\perp D$.  If
 $E_{P^\perp}$ is any conditional expectation of
 $\mathcal B(P^\perp\sH)$ onto $P^\perp D$, then
 $\bh \ni x \mapsto E_P(PxP)+ E_{P^\perp}(P^\perp xP^\perp)$ is a
 conditional expectation of $\bh$ onto $D$.  Therefore, there are
 multiple conditional expectations of $\bh$ onto $D$.  Since $D$ is an
 abelian von Neumann algebra, it is injective, and hence each
 conditional expectation of $\bh$ onto $D$ is a pseudo-expectation,
 see~\cite[Definition~1.3]{PittsStReInI}.  Since every regular MASA
 inclusion has a unique pseudo-expectation
 \cite[Theorem~3.5]{PittsStReInI} and the unique pseudo-expectation
 property is hereditary from
 above~\cite[Proposition~2.6]{PittsZarikianUnPsExC*In} it follows that
 $\bh$ is not an intermediate algebra for a regular MASA inclusion
 $(B,D)$.

 The case $P=I$ must be handled differently because in this case,
 there is a unique and faithful conditional expectation
 $E:\bh\rightarrow D$.  The proof will be accomplished in several
 steps.  Before embarking, note that we may assume $\sH=\ell^2(\bbN)$
 and that $D$ is the collection of all operators diagonal with respect
 to an orthonormal basis $\{\zeta_j\}_{j\in \bbN}$ for $\ell^2(\bbN)$.
Also, for non-zero vectors $\eta_1, \eta_2\in \sH$, we denote the rank-one
operator $\sH\ni \xi\mapsto \innerprod{\xi, \eta_2}\eta_1$ by
$\eta_1\eta_2^*$.

\dclaim $(\bh, D)$ is not regular.\footnote{Claim 1 was also
  established using very different methods by Katavolos and Paulsen
  in~\cite[Proposition~19]{KatavolosPaulsenOnRaBiPr}; we learned of
  their argument after finding the proof presented here.}  \enddclaim
\dproof
  By
 \cite[Example~3.10]{PittsZarikianUnPsExC*In}, 
$(\bh/\sK,  D/(D\cap \sK))$ is a MASA inclusion with a unique
pseudo-expectation, which is actually a conditional expectation
$\Delta: \bh/\sK\rightarrow D/(D\cap \sK)$.  As noted in
\cite[Example~3.10]{PittsZarikianUnPsExC*In}, $\Delta$ is not faithful.

Suppose now that $(\bh, D)$ is regular.  Then $(\bh/\sK, D/(D\cap\sK))$ is also
regular, so it is a regular MASA
inclusion.   By~\cite[Theorem~3.15]{PittsStReInI}, the left kernel of
$\Delta$ is an ideal  $\mathcal L\idealin\bh/\sK$.  But $\bh/\sK$ is simple,
so $\mathcal L=0$.   Hence $\Delta$ is faithful, yet as we have already
observed, it is not.  Thus $(\bh,D)$ is not regular. 
\enddproof

If $(B,D)$ is a regular MASA inclusion with
$D\subseteq \bh \subseteq B$, then there is a unique conditional
expectation of $B$ onto $D=I(D)$, \cite[Theorem~3.5]{PittsStReInI}.  The
point of the following is that we can reduce to the case where the
expectation of $B$ onto $D$ is faithful, that is, when $(B,D)$ is a \cstar-diagonal.
\dclaim Suppose $(B,D)$ is a
regular MASA inclusion with $D\subseteq \bh\subseteq B$.  Then there
exists a \cstar-diagonal $(C,D)$ with $D\subseteq \bh \subseteq C$.
\enddclaim
\dproof As $(B,D)$ is a regular MASA inclusion and $D$ is
injective, the unique pseudo-expectation $\tilde E$ for $(B,D)$
(see~\cite[Theorem~3.5]{PittsStReInI}) is a conditional expectation
whose restriction to $\bh$ is $E$.  We simplify notation and write
$E: B\rightarrow D$ instead of using $\tilde E$.

Since $D$ is an injective
\cstaralg,~\cite[Theorem~2.21]{PittsStReInI} shows that every pure state on $D$
extends uniquely to a pure state of $B$.   Let $\sL(B,D):=\{b\in B:
E(b^*b)=0\}$.  By~\cite[Theorem~3.15]{PittsStReInI},
$\sL(B,D)$ is
an ideal of $B$ having trivial intersection with $D$.  Then
$\sL(B,D)\cap \bh$ is an ideal in $\bh$ also having
trivial intersection with $D$.  But $\sL(B,D)\cap \bh\in \{\sK, (0)\}$,  so as
$\sK\cap D\neq (0)$, we conclude 
$\sL(B,D)\cap \bh=(0).$

Let $C=B/\sL(B,D)$ and let $q: B\rightarrow C$  be the quotient map.
By~\cite[Theorem~4.8]{DonsigPittsCoSyBoIs}, $(C,q(D))$ is a \cstar-diagonal.
Since $\sL(B,D)\cap \bh=(0)$, we may identify $(\bh,D)$ with
$(q(\bh),q(D))$, and doing so, we obtain $D\subseteq \bh\subseteq C$.  
\enddproof

\dclaim  Suppose $(B,D)$ is a \cstar-diagonal with $D\subseteq
\bh\subseteq B$.  Then $\sK$ is an essential ideal in $B$.
\enddclaim
\dproof
Let $n\in N(B,D)$ be non-zero and fix $j\in \bbN$ satisfying
$n\zeta_j\zeta_j^*\neq 0$.  We aim to show that
$v:=n\zeta_j\zeta_j^*$ belongs to $\bh$.  As
$v\in N(B,D)$, Corollary~\ref{dynamics}
shows there is a unique
$*$-isomorphism
$\theta_v: \overline{vv^*D}\rightarrow \overline{v^*v D}$ extending
the map $vv^*D\ni vv^*h\mapsto v^*hv$; further, for every
$h\in \overline{vv^*D}$, $v\theta_v(h)=hv$.

Since $v^*v=n^*n\zeta_j\zeta_j$, we see that $\overline{v^*v D} = v^*v
D=\bbC \zeta_j\zeta_j^*$.  Hence $\overline{vv^*D}$ is also
one-dimensional, so there exists some $k\in
\bbN$ so that
\[\overline{vv^*D}=vv^*D=\bbC \zeta_k\zeta_k^*.\]   This implies that
$\theta_v(\zeta_k\zeta_k^*)=\zeta_j\zeta_j^*$, hence
$v\theta_v(\zeta_k\zeta_k^*)=\zeta_k\zeta_k^*v$, that is,
\[
  v\zeta_j\zeta_j^*= \zeta_k\zeta_k^* v.\]

Now set $u=\zeta_k\zeta_j^*$.  Since
$u\zeta_j\zeta_j^*=\zeta_k\zeta_k^*u$, it follows that $u^*v$ commutes
with $D$.   As $D$ is a MASA in $B$, $u^*v\in D$.    A computation
shows $u^*vv^*u\neq 0$, so $u^*v\neq 0$.  Hence there is $0\neq c\in \bbC$
such that $u^*v=c\zeta_j\zeta_j^*$.  Then
\[u^*vv^*=c\zeta_j\zeta_j^*n^*.\] Taking adjoints and dividing by $c$
gives $n\zeta_j\zeta_j^*= c^{-1}vv^*\zeta_k\zeta_j^*$ showing that
$n\zeta_j\zeta_j^*$ is a multiple of a rank-one partial
isometry  in $N(\bh, D)$.

This holds for every $j\in \bbN$, so it follows that for any choice of
$i, j\in \bbN$,
\[n \zeta_i\zeta_j^*=n (\zeta_i\zeta_i^*)(\zeta_i\zeta_j^*)\in \sK.\]
As $\spn\{\zeta_i\zeta_j^*: i, j\in \bbN\}$ is dense in $\sK$, we see
that $n\sK\subseteq \sK$.   Replacing $n$ with $n^*$ in the arguments
above implies $\sK n\subseteq \sK$, so
\[\sK n \cup  n \sK\subseteq \sK.\]
Since $D$ is regular in $B$, we conclude that $\sK\idealin B$.  

Suppose $J\idealin B$ and $J\cap \sK=(0)$.  Arguing as in the second
paragraph of  the proof of Claim~2, $J\cap \bh =(0)$.  Let
$E:B\rightarrow D$ be the conditional expectation.  Taking $v=I$
in~\cite[Proposition~3.10]{DonsigPittsCoSyBoIs} we obtain
\[E(J)\subseteq J\cap D\subseteq J\cap \bh=(0).\]  The faithfulness of $E$ now gives $J=(0)$.
Thus $\sK\subseteq B$ is an essential ideal.
\enddproof

We now complete the proof of the part of Theorem~\ref{nointermed} assuming the sum
of the atoms of $D$ is the identity.  Arguing by contradiction,
suppose  $(B,D)$ is  a regular
MASA inclusion such that $D\subseteq \bh\subseteq B$.  Apply Claim~2
to obtain a \cstar-diagonal $(C,D)$ with $D\subseteq \bh\subseteq C$.
By Claim~3, $\sK$ is an essential ideal of $C$, so every element of
$C$ is an element of the multiplier algebra of $\sK$, that is,
$C\subseteq \bh$, so $C=\bh$.  Therefore, the inclusion $(\bh, D)=(C,D)$ is
regular, 
contradicting  Claim~1.

\end{proof}

It follows from the proof of Theorem~\ref{nointermed} that whenever
$D$ is a MASA in $\bh$, then $(\bh,D)$ is not regular.  
 It would be desirable to exhibit
an operator $T$ which does not belong to the norm
closure of the span of $N(\bh, D)$.

 When $(A,D)$ is a MASA inclusion with $A$ non-unital, some desirable properties of
$(A,D)$ (e.g. regularity or the unique state extension property) 
pass to $(\tilde A, \tilde D)$.  Using the notation and context of
part (2) of Corollary~\ref{regmap}, we now observe that 
this need not be the case when $(\tilde A,\tilde D)$ is replaced
with $(B, D_B)$, even when the original inclusion $(A,D)$ is
well-behaved.  Here is an example.

\begin{example}\label{essential}  
  Let $\sH=\ell^2(\bbN)$ and let $D$ be the set of all multiplication
  operators
  $\ell^2(\bbN)\ni (\zeta_n)_{n\in\bbN}\mapsto
  (d_n\zeta_n)_{n\in\bbN}$, where $(d_n)_{n\in \bbN}\in c_0$.  Then
  $(\sK, D)$ is a $C^*$-diagonal.

 We consider three  choices of a 
 unital algebra $B$ containing $\sK$ as an essential ideal.
 \begin{enumerate}
   \item Take $B=\tilde \sK$.  With this choice,  
$(B,D_B)=(\tilde \sK,\tilde D)$ is  a \cstar-diagonal.
\item Let $S$ be the unilateral shift
and let $B=C^*(S)$ be the Toeplitz algebra.  Then 
 $D_B=C^*(\{S^nS^{*n}: n\in
\bbN\}\cup \{I\})\simeq \tilde D$.
 Here $(B, D_B)$ is a Cartan pair, but not a \cstar-diagonal  because the unique
  state extension property fails.
\item 
Finally, take $B=M(\sK)=\bh$. Then 
$D_B\simeq \ell^\infty$
is an atomic MASA in $B$.  
By  Theorem~\ref{nointermed},
  the inclusion $(B,D_B)$ is not regular.  As is now well-known, the remarkable
paper~\cite{MarcusSpielmanSrivastavaInFaIIMiChPoKaSiPr} gives the
solution to the Kadison-Singer problem, so $(B, D_B)$ has 
the extension property.
\end{enumerate}
\end{example}

\begin{flexstate}{Problem}{}\label{gKS} Suppose $(A,D)$ is a
  \cstar-diagonal and  $B$ is a unital \cstaralg\ containing $A$ as an
  essential ideal.    Find conditions on $(A,D)$ and $B$ which ensure
  that 
  $(B, D_B)$ is regular, or that $(B,D_B)$ has the extension property.   
\end{flexstate}

Let $X$ be a locally compact, non-compact abelian group, fix
$x_0\in X$, let $\Gamma$ be the subgroup generated by $x_0$, and let
$\Gamma$ act on $C_0(X)$ by translation.    A class of \cstar-diagonals
which may
provide insight into Problem~\ref{gKS} are those of the form
$(C_0(X)\rtimes\Gamma, C_0(X))$.  

\def\cprime{$'$}
\providecommand{\bysame}{\leavevmode\hbox to3em{\hrulefill}\thinspace}
\providecommand{\MR}{\relax\ifhmode\unskip\space\fi MR }
\providecommand{\MRhref}[2]{%
  \href{http://www.ams.org/mathscinet-getitem?mr=#1}{#2}
}
\providecommand{\href}[2]{#2}

\end{document}